\documentclass{article}

\usepackage{etoolbox}
\usepackage{hyperref}

\makeatletter
\newcounter{author}
\renewcommand*\author[1]{%
  \stepcounter{author}%
  \ifnum\c@author=1
    \gdef\@author{#1}%
  \else
    \xdef\@author{\unexpanded\expandafter{\@author\and#1}}%
  \fi
  \csgdef{author@\the\c@author}{#1}}
\newcommand*\email[1]{%
  \csgdef{email@\the\c@author}{#1}}
\newcommand*\address[1]{%
  \csgdef{address@\the\c@author}{#1}}
\AtEndDocument{%
  \xdef\author@count{\the\c@author}%
  \c@author=1
  \print@authors}
\newcommand*\print@authors{%
  \ifnum\c@author>\author@count
  \else
    \print@author{\the\c@author}%
    \advance\c@author by 1
    \expandafter\print@authors
  \fi}
\newcommand*\print@author[1]{%
  \par\medskip
  \begin{tabular}{@{}l@{}}%
    \textsc{Addresses of \csuse{author@#1}}\\
    \csuse{address@#1}\\
    \textit{E-mail address}:
    \href{mailto:\csuse{email@#1}}{\csuse{email@#1}}
  \end{tabular}}
\makeatother

\author{Dongchen Jiao}
\address{Department of Mathematics, Brunel Univeristy London UB8 3PH}
\email{dongchen.jiao@brunel.ac.uk}

\author{Pascale Voegtli}
\address{Departement of Mathematics, University College London,WC1E 6BT}
\email{pascale.voegtli.20@ucl.ac.uk}

\usepackage{amsmath}
\usepackage{amssymb}
\usepackage{amsthm}
\usepackage{fancyhdr}

\newtheorem{theorem}{Theorem}[section]
\newtheorem{proposition}[theorem]{Proposition}
\newtheorem{definition}[theorem]{Definition}

\newtheorem{lemma}[theorem]{Lemma}

\newtheorem*{remark}{Remark}
\usepackage{tikz-cd}

\title{Flop connections between minimal models for corank 1 foliations over threefolds}

\begin{document}
\maketitle
\begin{abstract}
In 2007 Kawamata proved that two different minimal models can be connected by a sequence of flops. The aim of this paper is to show that the same holds true for 2 foliated minimal models descending from a common  3-fold pair equipped with a F-dlt foliation of corank 1.
\end{abstract}
\section{Introduction}
In recent years the understanding of the fundamental birational geometry of foliations, especially on 3-folds, has been promoted by several groundbreaking works. In \cite{cork1MMP} most parts of the classical MMP have been extended to 3-fold pairs equipped with a mildly singular corank 1 foliation. In particular, the existence of log-flips has been established.\par
The successful establishment of a foliated analogue of the classical MMP in low dimensions naturally raises the question whether classical results being closely related to the MMP do find their natural generalizations for foliated pairs.\par
One such classical result one might strive to convey to foliations is the well-known theorem of Kawamata \cite{Kawamata}, stating that two minimal models with terminal singularities are related by a sequence of flops.\par
Building upon the findings of the authors in \cite{cork1MMP}, the aim of this paper is to generalize the results of \cite{Kawamata} to foliated 3-folds equipped with a corank 1 foliation with F-dlt singularities.\par
In the proof of the main theorem (see below), we closely follow the line of reasoning found in \cite{Kawamata} and deviate only where adaptations or modifications seem inevitable.\par
Concretely, in the present publication, the following theorem is proven (For the precise definitions we refer to section (\ref{preliminaries})):
\begin{theorem}\label{theorem}
Let $(Y_{1}, \mathcal{F}_1)$ and $(Y_{2}, \mathcal{F}_2)$ be two foliated minimal models descending from a common foliated 3-fold pair $(X,\mathcal{F})$. We assume that $X$ is klt and $\mathbb{Q}$-factorial.\par
Assume there is a birational map $\alpha: Y_1 \dashrightarrow Y_2$, such that $K_{\mathcal{F}_1}$ and $K_{\mathcal{F}_2}$ are big. 
Then $\alpha$ is composed of a sequence of flops.\par
\end{theorem}
In detail, there are an effective $\mathbb{Q}$-divisor A on $Y_1$ such that $(\mathcal{F}_1,A)$ is F-dlt and a factorization
\begin{equation}
(Y_1,\mathcal{F}_1,A)=(X_0,\mathcal{G}_0,A_0)\dashrightarrow (X_1,\mathcal{G}_1,A_1)\dashrightarrow...\dashrightarrow (X_r,\mathcal{G}_r,A_r)=(Y_2,\mathcal{F}_2,A^{'})
\end{equation}
satisfying
\begin{enumerate}
\item $\beta_i: X_{i-1}\dashrightarrow X_{i}$ is a flip associated to a $(K_{\mathcal{G}_{i-1}}+A_{i-1})$-negative extremal ray, where $A_{i}$ is the strict transform of $A$ on $X_i$.
\item $\beta_i$ is crepant for $K_{\mathcal{G}_{i-1}}$ in the sense that pull-backs of $K_{\mathcal{G}_{i-1}}$ and $K_{\mathcal{G}_{i}}$ coincide on a common resolution.
\item the flipping contractions $\beta_i$ are $K_{\mathcal{G}_i}$-trivial.
\end{enumerate}
\section*{Acknowledgements}
We would like to express our deep gratitude to the PhD-advisor of the second author, Paolo Cascini, for the initialization of the project and his indispensable guidance throughout its further evolvement. Furthermore, are both authors indebted to Calum Spicer for his extended explanations, invaluable suggestions and careful revision of the first draft of this publication.\par
The first author is supported by EPSRC DTP studentship (EP/V520196/1) Brunel University London and would like to show appreciation to his supervisor Anne-Sophie Kaloghiros for useful discussion and invaluable suggestions.\par
The second author would like to thank the LSGNT for laying the mathematical foundations of this work by providing a stimulating and supporting working environment.
The second author was supported by the Engineering and Physical Sciences Research Council [EP/S021590/1]. 
The EPSRC Centre for Doctoral Training in Geometry and Number Theory (The London School of Geometry and Number Theory), University College London.
\section{Preliminaries}\label{preliminaries}
\subsection{Basic definitions}
Throughout, we work over the complex numbers and by 3-fold or variety we mean a complex analytic space of dimension 3 if not specified otherwise.\par
\begin{definition}
Let $X$ be a normal variety. A foliation $\mathcal{F}$ on $X$ is a coherent subsheaf of $\mathcal{T}_X$, such that
\begin{itemize}
    \item $\mathcal{F}$ is saturated, i.e. $\mathcal{T}_X/\mathcal{F}$ is torsion free and
    \item $\mathcal{F}$ is closed under Lie bracket.
\end{itemize}
\end{definition}
\begin{remark}
$\mathcal{F}$ is always locally free on the smooth locus $X_0$ of $X$.
\end{remark}
\begin{definition}
Suppose $\mathcal{F}$ is a rank $r$ foliation on a normal variety $X$. Notice that there exists an open embedding $j:X_0\hookrightarrow X$ such that $X_0$ is smooth,  $\mathrm{codim}(X\backslash X_0)\geq2$ and $\mathcal{F}$ is locally free on $X_0$, so $\wedge^r\mathcal{F}$ is an invertible sheaf on $X_0$. We define the canonical divisor of $\mathcal{F}$ to be any divisor $K_{\mathcal{F}}$ on $X$ such that $\mathcal{O}_X(-K_{\mathcal{F}})\cong j_*(\wedge^r\mathcal{F})$.
\end{definition}

\begin{definition}
By the notation $(X,\mathcal{F})$ we mean a variety $X$ equipped with a corank 1 foliation $\mathcal{F}$. 
\end{definition}
\begin{definition}
A triple $(X,D,\mathcal{F})$ consists of a foliation $\mathcal{F}$ on a variety X and an effective $\mathbb{R}$-divisor D such that $K_{\mathcal{F}} + D$ is $\mathbb{R}$-Cartier.
\end{definition}
For the definition of foliated minimal model we adopt the one found in [\cite{cork1MMP}, section 10]. For reader's convenience, we restate it here:\par
\begin{definition}
A minimal model of $\mathcal{F}$ is a $K_{\mathcal{F}}$-negative birational map $f: X \dashrightarrow X^\prime$
such that if $\mathcal{F}^\prime$ is the transformed foliation on $X^\prime$, then
\begin{enumerate}
\item $X^\prime$ is $\mathbb{Q}$-factorial and klt  
\item $\mathcal{F^\prime}$ is F-dlt and $K_{\mathcal{F}^\prime}$ is nef.
\end{enumerate}
\end{definition}
We now recall the definitions of foliation singularities as well as the notions of invariance, tangency rsp. transversality of subvarieties:
\begin{definition}
Let X be a normal variety and $\mathcal{F}$ a rank r foliation on X. A subvariety $S \subset X$ is called $\mathcal{F}$-\textbf{invariant} if for any open subset $U \subset X$ and any section $\partial \in H^{0}(U,\mathcal{F})$ we have:
\begin{equation}
\partial(I_{S \cap U}) \subset I_{S\cap U}
\end{equation}
where $I_{S\cap U}$ denotes the ideal sheaf of $S \cap U$ in U.\\
If $\Delta \subset X$ is a prime divisor, one defines $\epsilon(\Delta)=-1$ if $\Delta$ is invariant in the above sense and  $\epsilon(\Delta)=0$ otherwise.  
\end{definition}
\begin{definition}
Let X be a normal projective variety equipped with a non-dicritical corank 1 foliation $\mathcal{F}$.\par 
We call a subvariety $W \subset X$ \textbf{tangent} to $\mathcal{F}$ if for any birational morphism ${\pi: \tilde{X} \to X}$ and any divisor E on X such that $E$ dominates $W$, we have that E is $\mathcal{\tilde{F}}$-invariant, where $\mathcal{\tilde{F}}$ denotes the pulled back foliation on $\tilde{X}$.\par
Otherwise, we call $W\subset X$ transverse to $\mathcal{F}$.
\end{definition}
\begin{definition}
For a birational morphism $\pi: \tilde{X} \to X$ and a foliated pair $(\mathcal{F},\Delta)$ on $X$, let $\tilde{\mathcal{F}}$ be the pulled back foliation on $\tilde{X}$ and $\tilde{\Delta}$ the strict transform of $\Delta $ on $\tilde{X}$. We then write:
\begin{equation}
K_{\mathcal{\tilde{F}}} + \tilde{\Delta} = \pi^{*}(K_{\mathcal{F}} + \Delta) + \Sigma a(E,\mathcal{F},\Delta)E
\end{equation}
where $\pi_{*}K_{\mathcal{\tilde{F}}}= K_{\mathcal{F}}$ and the sum runs over all prime exceptional divisors on $\tilde{X}$. The rational numbers $a(E,\mathcal{F},\Delta)$ are called discrepancies.\par
Given a normal variety X and a foliated pair $(\mathcal{F},\Delta)$ on X we call $(\mathcal{F},\Delta)$ terminal (rsp. canonical, log-canonical) if $a(E,\mathcal{F},\Delta) > 0$ (rsp.$\geq0,\geq -\epsilon(E)$) for every exceptional prime divisors E.\par
For a not necessarily closed point $P \in X$ we say $(\mathcal{F},\Delta)$ is terminal (rsp. canonical, log-canonical) at P if for all birational morphisms $\pi: \tilde{X}\to X$ and any $\pi$-exceptional divisor E on $\tilde{X}$ whose center lies in the Zariski closure $\bar{P}$ of P we have that the discrepancy of E is $>0 (\geq 0, rsp. \geq - \epsilon(E))$.\par
\end{definition}
There are yet another two important concepts related to singularities that will be needed in the following: For reader's convenience we recall here the definition of F-dlt-singularities and the notion of log-smoothness.\par
\begin{definition}
Let X be a normal variety and $\mathcal{F}$ a corank 1 foliation on X. Let A be a $\mathbb{Q}$-divisor such that $(K_{\mathcal{F}}+A)$ is $\mathbb{Q}$-Cartier.\par
We call $(\mathcal{F},A)$ foliated divisorial log terminal (F-dlt) if
\begin{enumerate}
\item Each irreducible component of A is generically transverse to the foliation $\mathcal{F}$ and has coefficients at most one
\item There is a foliated log resolution $\pi: Y \rightarrow X$ of $(\mathcal{F},A)$ which only extracts divisors E of log discrepancy $a(E,\mathcal{F},A)>-\epsilon(E)$.
\end{enumerate} 
\end{definition}
\begin{definition}
Given a germ $p\in X$ with a foliation $\mathcal{F}$ such that $p$ is a singular point of $\mathcal{F}$ we call a (formal) hypersurface germ $p\in S$ a (formal) separatrix if it is invariant under $\mathcal{F}$.
\end{definition}
\begin{definition}
Given $(X,\mathcal{F})$ as in the previous definition, we say that $(\mathcal{F},A)$ is foliated log smooth if:
\begin{enumerate}
\item $(X,A)$ is log smooth
\item $\mathcal{F}$ has simple singularities
\item If S denotes the support of non-$\mathcal{F}$-invariant components of A, $p \in S$ is a closed point and $\Sigma_1,...,\Sigma_k$ are $\mathcal{F}$-invariant divisors passing through p, then $\Sigma_1 \cup...\cup\Sigma_k$ is a normal crossings divisor at p.
\end{enumerate}
\end{definition}
For the precise definition of simple singularities, the inaugurated reader may again wish to consult [\cite{HDFM}, Def. 2.13].
Next, we clarify the notions of log flips and flops:
We state them in classical terms, i.e in the way they arise in the realm of the classical MMP. The respective generalizations to the foliated situation are though straight forward ( simply replace  $K_X$ by $K_{\mathcal{F}}$).
The definitions are verbatim taken from [\cite{km}, Def. 3.33 \textnormal{rsp} 6.10]
\begin{definition}
Let X be a normal scheme and D a $\mathbb{Q}$-divisor on X such that $K_X + D $ is $\mathbb{Q}$-Cartier.\\
A $(K_X + D)$-flipping contraction is a proper birational morphism $f: X \rightarrow Y$ to a normal scheme Y such that $Exc(f)$ has codimension at least two in X and -$(K_X + D)$ is f-ample.
A normal scheme $X^{+}$ together with a proper birational morphism $f^{+}: X^{+} \rightarrow Y$ is called a $(K_X + D)$-flip of f if:
\begin{enumerate}
\item  $K_{X^{+}} + D^{+}$ is $\mathbb{Q}$-Cartier, where $D^{+}$ is the birational transform of D on $X^{+}$
\item $K_{X^{+}} + D^{+}$ is $f^{+}$-ample
\item $Exc(f^{+})$ has codimension at least two in $X^{+}$
\end{enumerate}
\end{definition}
\begin{remark}
By an abuse of notation we also call the map $\phi: X \dashrightarrow X^{+}$ a $(K_X+D)$-flip or simply a D-flip. In a analogous way we can define a $K_{\mathcal{F}}+D$-flip for a foliation $\mathcal{F}$ on $X$.
\end{remark}
\begin{definition}
Let X be a normal scheme with klt singularity. Let $\mathcal{F}$ be a F-dlt foliation on $X$. A flopping contraction is a proper birational morphism $f: X \rightarrow Y$ to a normal scheme Y such that $Exc(f)$ has codimension at least two in X and $K_\mathcal{F}$ is numerically trivial.
In the above setup, i.e assuming $K_\mathcal{F}$ numerically trivial: if D is a $\mathbb{Q}$-Cartier divisor on X such that $-(K_\mathcal{F} +D)$ is f-ample, then the D-flip of f is also called the D-flop.\par
We recall that a divisor F on X is called numerically f-trivial for the birational contraction morphism $f: X \rightarrow Y$ if for every curve $C$ contracted by $f$, we have $F\cdot C=0$ .
\end{definition}
In the sequel of the proof of theorem \ref{theorem} we will make use of the following two results from \cite{cork1MMP}. We restate them here verbatim but refer to the original publication for a proof of the statements:
\begin{theorem} [\cite{cork1MMP},Theorem 9.4] \label{bpf}
Let X be a normal projective 3-fold with klt singularities. Let $\mathcal{F}$ be a corank 1 foliation on X. Let $\Delta$ be a $\mathbb{Q}$-divisor such that $(\mathcal{F},\Delta)$ is a F-dlt pair. Let $A \geq0$ and $B \geq 0$ be $\mathbb{Q}$-divisors such that $\Delta=A+B$ and A ample. Assume $K_\mathcal{F} +\Delta$ is nef.\par
Then $K_\mathcal{F} +\Delta$ is semi-ample. 
\end{theorem}
\begin{theorem}[\cite{cork1MMP},Lemma 3.24] \label{Bertini}
Let X be a normal projective 3-fold, $\mathcal{F}$ a corank1 foliation on X . Let $(\mathcal{F},\Delta)$ be an F-dlt pair such that $\lfloor \Delta \rfloor =0$  and let A be an ample $\mathbb{Q}$-divisor on X.\par
Then there is an effective $\mathbb{Q}$-divisor $A^{'} \sim_{Q}A$ such that:
\begin{enumerate}
\item $(\mathcal{F},\Delta)$ is also F-dlt
\item $\lfloor \Delta + A^\prime \rfloor =0$
\item The support of $A^\prime$ does not contain any log canonical center of $(\mathcal{F},\Delta)$
\end{enumerate}
\end{theorem}
\section{Proof of the Theorem \ref{theorem}}
Sticking to the notation in the statement of Theorem \ref{theorem}, we first show that the birational map $\alpha$ relating the 2 foliated minimal models is a small map, i.e the codimension of the exceptional locus is at least 2.\par
We start by setting up the notation:
Let  $(Y_1,\mathcal{F}_1)$ and $(Y_2,\mathcal{F}_2)$ be 2 foliated minimal models of a common log-canonical foliated threefold pair 
$(X,\mathcal{F})$, such that $X$ is $\mathbb{Q}$-factorial with klt singularities. Denote by 
$\alpha: (Y_1,\mathcal{F}_1)\dashrightarrow (Y_2,\mathcal{F}_2)$ the map connecting the 2 minimal models and by $\alpha_i: (X,\mathcal{F})\dashrightarrow (Y_i,\mathcal{F}_i)$ the sequence of MMP-steps run through in order to obtain the respective minimal models.\par
We recall that the steps of the foliated MMP are $K_{\mathcal{F}}$-negative and hence so are the above defined $\alpha_i's$. Thus there is a common log resolution $W$ of $X$, $Y_1$ and $Y_2$ such that for the following commutative diagram we have:
\begin{center}
\begin{tikzcd} \label{setup}
W
\arrow[drr, bend left, "q_{2}"]
\arrow[ddr, bend right, "q_{1}"]
\arrow[dr, "p"] & & \\
& (X,\mathcal{F}) \arrow[r, dotted, "\alpha_2"] \arrow[d, dotted, "\alpha_1"]
& (Y_2,\mathcal{F}_1) \\
& (Y_1,\mathcal{F}_2)\arrow[ur,dotted,"\alpha",swap]\\
\end{tikzcd}
\end{center}
\begin{equation}\label{eq1}     
p^{*}(K_{\mathcal{F}})=q_{1}^{*}(K_{\mathcal{F}_1})+E_1= q_{2}^{*}(K_{\mathcal{F}_2}) + E_2
\end{equation}\\
where $E_i\geq 0$, $q_{i}-exceptional$ and $p_{*}^{-1}Exc(\alpha_i) \subset supp E_i$
\begin{lemma}
    With the above notation, we have that $\alpha$ is small.
\end{lemma}
\begin{proof}
    Suppose $E_1\neq E_2$ then without loss of generality we assume that $\overline{E_1}:=E_1-\mathrm{min}\{E_1,E_2\}>0$, $\overline{E_2}:=E_2-\mathrm{min}\{E_1,E_2\}\geq0$. According to the negativity lemma we can find a curve $C\subseteq\mathrm{Supp}\ \overline{E_1}$ such that $C$ is not contained in $\mathrm{Supp}\ \overline{E_2}$ and $C\cdot\overline{E_1}<0$. Then we get 
    $$q_2^*(K_{\mathcal{F}_2})=q_1^*(K_{\mathcal{F}_1})+E_1-E_2=q_1^*(K_{\mathcal{F}_1})+\overline{E_1}-\overline{E_2}$$
    and
    $$q_2^*(K_{\mathcal{F}_2})\cdot C=q_1^*(K_{\mathcal{F}_1})\cdot C+(\overline{E_1}-\overline{E_2})\cdot C<0$$
    contradicting the nefness of $\mathcal{F}_i$.\par 
    Hence we conclude that $\overline{E_1}=\overline{E_2}$ which implies that  $q_1$ and $q_2$ contract the same divisors, thus $\alpha$ is small. 
\end{proof}
The overall strategy of Kawamata in the analogous classical statement, i.e. for minimal models of terminal varieties instead of foliated minimal models as in the present paper, is to choose an ample divisor $L^\prime$ on minimal model $Y_2$ and consider its strict transform $L$ on minimal model $Y_1$. As klt-ness is an open property for varieties $(Y_1,lL)$, for a small number $l$, is still klt and Kawamata can thus assume that the divisor $K_{Y_1}+lL$ is not nef because otherwise $\mathbb{Q}$-factoriality of $Y_2$ and the classical base point free theorem (\cite{km}, Theorem 3.3) imply that $\alpha$ is an isomorphism. This allows Kawamata then to run a $(K_{Y_1}+lL)$-MMP on $Y_1$ to finally reach his conclusions.\par
In the foliated case treated here, although there is as well an analogue of the classical basepoint-free theorem for corank1-foliations on 3-folds (see section \ref{preliminaries}), its invocation is not as straightforward as is in the original proof. The main difficulties arise from the requirement in the foliated version of the basepoint-free theorem that $(\mathcal{F}_1,lL)$ must be F-dlt. A priori it is not immediate that this condition can be satisfied. \par
The core of the present paper thus consists in the demonstration that a careful choice of the ample divisor $L^\prime$ on $Y_2$ guarantees that $(\mathcal{F}_1,lL)$ becomes F-dlt.\par
Invoking (\ref{Bertini}), we know that we can find an ample $L^\prime$  on $Y_2$ and a sufficiently small real number l such that $(\mathcal{F}_2,lL^\prime)$ is again F-dlt. The next few Lemmata in the present paper lay the ground for the proof that $L^\prime$ can be chosen such that also its strict transform on $Y_1$ satisfies $(\mathcal{F}_1,lL)$ F-dlt.
\begin{remark}
Notice that the Bertini-type theorem \ref{Bertini} does not directly apply to $(\mathcal{F}_1,lL)$ as $L$ is no longer ample.
\end{remark} 
\subsection{Exceptional curves are foliated-trivial}
In this section we are going to prove the following proposition:
\begin{proposition}\label{KFtriviality}
For any curve $C$ lying in the exceptional locus of $\alpha$, we have $K_{\mathcal{F}_1}\cdot C=0$.
\end{proposition}
We will prove this in two steps:
\begin{lemma}\label{tri}
Suppose $K_{\mathcal{F}_1}$ is big and nef and $K_{\mathcal{F}_1}\cdot C>0$, then there exists $0<\epsilon\ll1$, such that
$$C\nsubseteq\textup{\textbf{B}}_+(K_{\mathcal{F}_1}+\epsilon L)$$
\end{lemma}
\begin{proof}
Fix an ample divisor $A$ on $Y_1$. Since $K_{\mathcal{F}_1}$ is nef and $K_{\mathcal{F}_1}\cdot C>0$, according to [\cite{birkar2017augmented}, Theorem 1.3] we have
$$C\nsubseteq\mathrm{\textbf{B}}_+(K_{\mathcal{F}_1})$$
i.e. for any fixed $n\gg0$, we have
$$C\nsubseteq\mathrm{\textbf{B}}(K_{\mathcal{F}_1}-\frac{1}{n}A)$$ Then we can find $0\leq H\sim_\mathbb{Q}K_{\mathcal{F}_1}-\frac{1}{n}A$ such that $C\nsubseteq H$. We thus have
$$K_{\mathcal{F}_1}+\epsilon L=K_{\mathcal{F}_1}-\frac{1}{n}A+\frac{1}{n}A+\epsilon L\sim_\mathbb{Q}H+(\frac{1}{n}A+\epsilon L)$$
Since $A$ is ample, for $0<\epsilon\ll1$ we have $\frac{1}{n}A+\epsilon L$ is also ample, and so is $\frac{1}{n}A+\epsilon L-\delta A$ for some $0<\delta\ll1$ ,so we may find an effective $\mathbb{Q}$-divisor $H_0\sim_\mathbb{Q}K_{\mathcal{F}_1}+\epsilon L-\delta A$ such that $C\nsubseteq H_0$. 
\end{proof}
\begin{lemma}\label{triv2}
There exists an ample divisor $L^\prime$ such that $\mathrm{Exc}(\alpha)=\textup{\textbf{B}}_+(K_{\mathcal{F}_1}+\epsilon L)$ for $L=\alpha_*^{-1}L^\prime$ and $\epsilon$ from the above proposition.
\end{lemma}
\begin{proof}
First we choose an ample divisor $L^\prime$ on $Y_2$ with $0<\epsilon\ll1, \epsilon\in\mathbb{Q}$ such that $C\nsubseteq\textup{\textbf{B}}_+(K_{\mathcal{F}_1}+\epsilon L)$.\par 
According to [\cite{boucksom2014augmented}, Theorem A], we have
$$\mathrm{Exc}(\Phi_m)=\textup{\textbf{B}}_+(K_{\mathcal{F}_1}+\epsilon L)$$
where $\Phi_m: Y_1\dashrightarrow H^0(Y_1,m(K_{\mathcal{F}_1}+\epsilon L))$ is the map induced by the complete linear system $|m(K_{\mathcal{F}_1}+\epsilon L)|$ on $Y_1$ with $m\gg0$ and sufficiently divisible.\par
However, since the $\mathbb{Q}$-divisor $K_{\mathcal{F}_2}+\epsilon L^\prime$ is ample, we assume that $n(K_{\mathcal{F}_2}+\epsilon L^\prime)$ is very ample and since $\alpha$ is small, $\alpha$ is the rational map induced by the complete linear system
$$|n(K_{\mathcal{F}_1}+\epsilon L)|=\alpha_*^{-1}|n(K_{\mathcal{F}_2}+\epsilon L^\prime)|$$
hence $\Phi_{mn}=p\circ\alpha$ where $p$ is the twisted embedding $H^0(Y_1,n(K_{\mathcal{F}_1}+\epsilon L))\hookrightarrow H^0(Y_1,mn(K_{\mathcal{F}_1}+\epsilon L))$. So $\mathrm{Exc}(\alpha)=\mathrm{Exc}(\Phi_{mn})$.
\end{proof}
Now we combine Lemma \ref{tri} and Lemma \ref{triv2} and one can easily see that Proposition \ref{KFtriviality} holds. 
\subsection{An exceptional curve is not an lcc for the foliation}
\begin{lemma}\label{C not an lcc}
Notation as above. There is no curve contained in
$${Exc(\alpha):=\{ x \in Y_1| \text{$\alpha$ is not an isomorphism around}\ x\}}$$ 
which is a log canonical center for the foliation $\mathcal{F}_1$.
\end{lemma}
\begin{proof} \label{not lcc}
Assume for sake of contradiction there was a curve $C \subseteq \mathrm{Exc}(\alpha)$ whose generic point was a log canonical foliation singularity. By [\cite{HDFM}, Definition 2.24], passing to a foliated log resolution $W$, there is a divisor $E$ such that $a(E,\mathcal{F}_1)=-\epsilon(E)$. Notice that the $K_\mathcal{F}$-negativity of the map $\alpha_1$ gives rise to the following equation:
\begin{equation}\label{negativity}
p^{*}K_{\mathcal{F}}= q_{1}^{*}K_{\mathcal{F}_1} + E_1
\end{equation}
such that $p_{*}^{-1}\mathrm{Exc}(\alpha_1)\subset\mathrm{Supp}E_1$ and $E_1\geq0$.\par
First notice that as $C \subset \mathrm{Exc}(\alpha)$ we know that any exceptional divisor on the log resolution $W$ is either $\alpha_1$ or $\alpha_2$-exceptional. \par
 We can distinguish two scenarios. Either $E$ descends to $X$ as a divisor, or $center_X(E)$ is contained in the flipping locus of some flip arising in the decomposition of $\alpha_2$ into flips and divisorial contractions.\par
In the first case, $E$ descends to a divisor $E_X$ on $X$ with the property that $\alpha_1(E_X)=C$ and which is of discrepancy $a(E,F_1,\Delta_1)=-\epsilon(E_X)$ as the restriction of $p$ to $E$ in this case is an isomorphism and hence preserves discrepancies.\par
On the other hand, from equation (\ref{negativity}) we infer that $p_{*}^{-1}(E_X)=E \subset E_1$. This is however impossible as $a(E,\mathcal{F}_1)=-\epsilon(E)\leq 0$ whereas $E_1>0$ by construction.\par
Consequently, we may assume that $D:=center_X(E)$ is not divisorial. But then $C$ is a flipped curve for the map $\alpha_2$ (as $V$ in this case is an isomorphism by and hence cannot be a log canonical center by (\cite{cork1MMP}, Lemma 2.7)). The preceding comment in brackets though suggests that $X$ and $X_1$ are locally isomorpic around the curve $C$, which implies together with the local nature of flips that the flip of $D$ in the course of the MMP defined by $\alpha_2$ could also be carried out on $X_1$ contradicting the fact that $X_1$ is minimal. Hence the claim. 
\end{proof}
\subsection{The foliation is terminal along the curve}
\begin{lemma}\label{tangency}
A curve $C$ in $\mathrm{Exc}(\alpha)$ satisfying $K_{\mathcal{F}}\cdot C=0$ is tangent to the foliation $\mathcal{F}$.
\end{lemma}
\begin{proof}
As $K_{\mathcal{F}}$ is big there exist an ample divisor $A$ and an effective divisor E such that $K_{\mathcal{F}} \sim_{Q} A+E$. By assumption we furthermore have $0=K_{\mathcal{F}}\cdot C=(A+E)\cdot C$, which implies that $E\cdot C<0$ due to ampleness of $A$. Let $L$ be a component of E such that $L\cdot C<0$, i.e. $C$ is contained in $\mathrm{Supp}(L)$. By the lemma (\ref{C not an lcc}) we know that $C$ is not an lcc for $\mathcal{F}$ and hence can find $\epsilon>0$ such that $C$ defines an lcc for $\mathcal{F}+ \epsilon L$. Notice that by construction we have $(K_{\mathcal{F}}+ \epsilon L)\cdot C<0$.
We can now invoke [\cite{HDFM},Thm 4.5] to conclude that $C$ needs to be tangent as claimed.
\end{proof}
Lemma (\ref{C not an lcc}) proved that $C$ itself is not a log canonical center of the foliation $\mathcal{F}$, we next want to show that in addition, there are no 0-dimensional log canonical centers located along $C$. Together with the above demonstrated tangency of $C$ this will allow us to conclude that the foliation $\mathcal{F}$ is terminal along $C$.\par
\begin{lemma}
 Let $C$ be a tangent curve contained in $\mathrm{Exc}(\alpha)$, then $C$ does not contain any lc center of $\mathcal{F}$.
\end{lemma}
\begin{proof}
As $C$ is tangent by assumption, we can invoke (\cite{HDFM},  Thm 5.11) to deduce that there is a germ of an analytic surface $S$ such that $C \subset S$ and S is foliation invariant.\par 
For sake of contradiction we next assume that there is a 0-dimensional log canonical center located along $C$. By [\cite{cork1MMP},thm 3.8] we can deduce that $\mathcal{F}$ is foliated log smooth around $p$.\par 
Notice that in this situation all log canonical centers  are precisely given by the strata of Sing($\mathcal{F}$) around the simple singularity of the foliation. \par
Furthermore, the local description of simple singularities ( for definition see [\cite{cork1MMP}, Def. 2.8] in order for $p$ to be a log canonical center, we must have at least two one dimensional strata of Sing($\mathcal{F}$), corresponding to the intersection lines of two formal coordinate hyperplane germs with the surface $S$, intersecting $C$ in the point $p$. The scenario is depicted in the picture (\ref{2lcc}) below where we labelled the 2 presumptive 1-dimensional log canonical centers meeting $C$ by $\xi_1$ and $\xi_2$. We next argue that the depicted picture cannot occur.\par 
According to Lemma \ref{tri} we know that we can assume $K_\mathcal{F} \cdot C=0$. Thus in the described scenario, the following calculation would hold true:
\begin{equation} \label{xi}
0= K_\mathcal{F}\cdot C= (K_S + \sum(\xi_i))\cdot C
\end{equation}
where the sum runs over strata of Sing($\mathcal{F}$) intersecting in $C$ ( in (\ref{2lcc}) we only drew 2 of them for the sake of clarity). In case that there are at least two we conclude that $-K_S \cdot C \geq 2$. As the following reasoning displays, this would however imply that $C$ deforms which is impossible as $C$ by assumption is contained in the one dimensional $\mathrm{Exc}(\alpha)$. In order to get the contradiction, we need to assume that $C$ is rational, an assumption which will be justified right after.\par
Here we regard $S$ as an analytic surface. And for simplicity we assume that $S$ is smooth here. The case when $S$ is singular will be discussed later.\par
According to standard deformation theory, an estimate on the dimension of the deformation space of a morphism $f$ from a curve $C$ to a surface $S$ is given as follows:
\begin{equation}\label{deform}
\dim_{[f]}\mathrm{Mor}(C,S)\geq \chi(C,f^{*}T_S) = -K_{S}.f^{*}C +2(1-g(C))
\end{equation}
where the equality follows from Hirzebruch-Riemann-Roch \cite{Debarre}.\par
Assuming thus for the moment that $g(C)=0$, our previous calculations show that $$\mathrm{dim}_{[f]}\mathrm{Mor}(C,S)\geq 4$$ 
Since $\dim\mathrm{Aut}(C)\leq3$, $C$ would deform, a contradiction.\par
Now we turn back to the general case when $S$ is singular. In this case we take a minimal resolution $\varphi:S^\prime\to S$. Since $S$ is not smooth hence not terminal we have
$$K_{S^\prime}+E=\varphi^*K_S$$
for some effective exceptional divisor $E$ on $S^\prime$. According to the simple singularity condition we have that $S$ is smooth at the generic point of $C$. So if we take $C^\prime=\varphi_*^{-1}C$ we have
$$K_{S^\prime}\cdot C^\prime\leq K_S\cdot C\leq-2$$
Then the argument (\ref{deform}) still works.\par
We are then left to justify the assumption that $C$ is of arithmetic genus 0.\par
In order to verify this claim it suffices to show that its normalization is a smooth rational curve. Using  equation (\ref{xi}) and taking advantage of the projection formula as well as generalized adjunction results (see [\cite{KollarSing}, \text{chapter 4}]) we obtain the following sequence of equalities
\begin{equation}
\deg(\nu)(K_\mathcal{F} -\sum(\xi_i))|_C + \tilde{C}^2 =K_S.
\cdot\tilde{C}+\tilde{C}^2 \sim_{\mathcal{Q}} K_{\tilde{C}}+ \mathrm{Diff}_{\tilde{C}}
\end{equation} 
where $\nu:\tilde{C} \mapsto C$ denotes the normalization and $\mathrm{Diff}_{\tilde{C}}$ the different of the adjunction.\par
Now notice that the normal bundle of $\tilde{C}$, given by $\tilde{C}^2$ has non-positive degree as $C$ does not deform. Thus, the whole left hand side of the $\mathbb{Q}$-linear equivalence above is of non-positive degree and in turn so is the right hand side. Noticing that $\mathrm{Diff}_{\tilde{C}}$ is effective, we conclude that $K_{\tilde{C}}$ has non-positive degree which implies that $\tilde{C}$ and hence $C$ are genus 0 curves.\par
Summing up all the gained insight about the local situation around the presumptive log canonical centre $p$ located on $C$, the only conceivable picture could be the one depicted in (\ref{lcc1}).
However, as the log canonical locus of the foliation is given precisely by strata of Sing($\mathcal{F}$) and codimensionality of a simple singularity is given by the number of maximal dimensional strata it is contained in, $p$ being 0-dimensional cannot be a log canonical center of $\mathcal{F}$.\par
 
\begin{figure}
\centering
\begin{tikzpicture}
\draw (0,0)rectangle (4,4) node[anchor=north west] {$S$};
\draw (0,1)..controls (2,1) and (3,1.5)..(4,2) node[anchor=north west] {$C$};
\draw (2.25,0)..controls (2,1) and (2.5,3)..(3,4) node[anchor=north] {$\xi_1$};
\draw (3,0)..controls (2,1) and (1.5,3)..(2.25,4) node [anchor=north] {$\xi_2$};

\end{tikzpicture}
\caption{No more than 2 log canonical centers $\xi_1$ and $\xi_2$ meeting $C$ in a common point $p$}
\label{2lcc}
\end{figure}
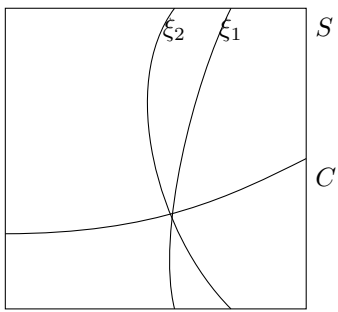

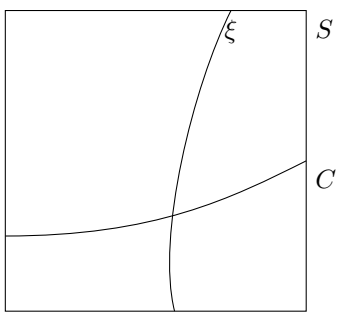
\begin{figure}
\centering
\begin{tikzpicture}
\draw (0,0)rectangle (4,4) node[anchor=north west] {$S$};
\draw (0,1)..controls (2,1) and (3,1.5)..(4,2) node[anchor=north west] {$C$};
\draw (2.25,0)..controls (2,1) and (2.5,3)..(3,4) node[anchor=north] {$\xi$};
\end{tikzpicture}
\caption{1 lcc meeting $C$ in a common point $p$}
\label{lcc1}
\end{figure}
\end{proof}
\begin{proof}[proof of Theorem \ref{theorem}]
Altogether, we have demonstrated that $\mathcal{F}$ is terminal along any curve $C\subset\mathrm{Exc}(\alpha)$. We claim that this now allows us to find $l > 0$ and an ample divisor $L^\prime$ on $Y_2$ satisfying the conditions in theorem (\ref{Bertini}) such that the strict transform $ L:= \alpha_{*}^{-1}L^\prime$ fulfills the desired condition that $(Y_1,\mathcal{F}_1,l L)$ is a F-dlt pair.\par
Herefore notice, that the proof of (\ref{Bertini}) in [\cite{cork1MMP}, Lemma 3.24] essentially boils down to the demonstration of bullet points $(i)-(iv)$ in the notation of the original publication. This in turn implies that the theorem can be applied to L directly in case one is able to establish these 4 conditions. We now shortly comment on them individually justifying our claim : \par
Notice first that outside the exceptional locus $\mathrm{Exc}(\alpha)$, $L$ is isomorphic to $L^\prime$ and the conditions are hence satisfied automatically. Furthermore, there are no log canonical centers located on $\mathrm{Exc}(\alpha)$, so (i) is established. Point (ii) is automatically satisfied as $\dim\mathrm{Exc}(\alpha)=1$. For (iii), again, there are no log canonical centers located on $\mathrm{Exc}(\alpha)$. Lastly, ad (iv) notice that as $\mathcal{F}$ is terminal along $ C\subset\mathrm{Exc}(\alpha)$ as we have demonstrated, we can find a real number l small enough such that any exceptional divisor E with center contained in $\mathrm{Exc}(\alpha)$ will satisfy $a(E,\mathcal{F},lL) \geq \epsilon(E)$. 
We thus conclude by theorem (\ref{bpf}) that the F-dlt pair $(\mathcal{F},lL)$ would be semiample in case it was nef. By the same logic as in the original publication by Kawamata this would however force the small map $\alpha$ to be an isomorphism and hence this case can be disregarded.\par
We can thus assume in the sequel that $K_{\mathcal{F}_1} +lL$ as constructed above is not nef and run a $(K_{\mathcal{F}_1} + lL)$-MMP in complete analogy to Kawamata. Where it is to be stressed that by results of the authors of \cite{cork1MMP}, the F-dlt property is preserved under a $K_{\mathcal{F}}$-MMP.
It is also worth mentioning that this $(K_{\mathcal{F}_1} + lL)$-MMP which by construction consists of a sequence of flips is actually $K_{\mathcal{F}_1}$-trivial as has been demonstrated above and need not, as opposed to the classical case in Kawamata's publication, be deduced from a clever application of the cone theorem.\par
Last, according to Proposition \ref{KFtriviality}, every flip $\beta_i: X_{i-1}\dashrightarrow X_{i}$ is a $K_{\mathcal{G}_{i-1}}$-flop since every curve we contract is $K_{\mathcal{G}_{i-1}}$-trivial. This concludes the proof of theorem \ref{theorem}.
\end{proof}
\bibliographystyle{plain} 
\bibliography{ref}

\begin{thebibliography}{1}

\bibitem{birkar2017augmented}
Caucher Birkar.
\newblock The augmented base locus of real divisors over arbitrary fields.
\newblock {\em Mathematische Annalen}, 368(3-4):905--921, 2017.

\bibitem{boucksom2014augmented}
S{\'e}bastien Boucksom, Salvatore Cacciola, and Angelo~Felice Lopez.
\newblock Augmented base loci and restricted volumes on normal varieties.
\newblock {\em Mathematische Zeitschrift}, 278(3-4):979--985, 2014.

\bibitem{cork1MMP}
Paolo Cascini and Calum Spicer.
\newblock \text{MMP} for co-rank one foliations on threefolds.
\newblock {\em Inventiones mathematicae}, 225(2):603--690, 2021.

\bibitem{Debarre}
Olivier Debarre.
\newblock {\em Higher-Dimensional Algebraic Geometry}.
\newblock Springer New York, NY, 2013.

\bibitem{Kawamata}
Yujiro Kawamata.
\newblock Flops connect minimal models.
\newblock {\em Publications of the research institute for mathematical
  sciences}, 44(2):419--423, 2008.

\bibitem{km}
Janos Kollár and Shigefumi Mori.
\newblock {\em Birational Geometry of Algebraic Varieties}.
\newblock Cambridge Tracts in Mathematics. Cambridge University Press, 1998.

\bibitem{KollarSing}
János Kollár.
\newblock {\em Singularities of the Minimal Model Program}.
\newblock Cambridge Tracts in Mathematics. Cambridge University Press, 2013.

\bibitem{HDFM}
Calum Spicer.
\newblock Higher-dimensional foliated mori theory.
\newblock {\em Compositio Mathematica}, 156(1):1--38, 2020.

\end{thebibliography}
\end{document}